\documentclass[12pt]{amsart}

\usepackage{amssymb}
\usepackage{mathrsfs}

\usepackage{enumitem}
\usepackage{xcolor}
\usepackage{graphicx}

\makeatletter
\@namedef{subjclassname@2010}{%
  \textup{2010} Mathematics Subject Classification}
\makeatother

\usepackage[T1]{fontenc}

\newtheorem{theorem}{Theorem}[section]
\newtheorem{corollary}[theorem]{Corollary}
\newtheorem{lemma}[theorem]{Lemma}
\newtheorem{proposition}[theorem]{Proposition}

\theoremstyle{definition}
\newtheorem{definition}[theorem]{Definition}
\newtheorem{remark}[theorem]{Remark}
\newtheorem{example}[theorem]{Example}
\newtheorem{question}[theorem]{Question}

\numberwithin{equation}{section}

\newcommand*{\Le}{\leqslant}
\newcommand*{\Ge}{\geqslant}
\newcommand{\inp}[2]{\langle{#1},\,{#2} \rangle}
\newcommand{\beqn}{\begin{eqnarray*}}
\newcommand{\eeqn}{\end{eqnarray*}}
\newcommand{\beq}{\begin{eqnarray}}
\newcommand{\eeq}{\end{eqnarray}}
\newcommand*{\wot}[1]{\mathscr W_{#1}}

\def \H{\mathcal{H}}
\def \D{\mathbb{D}}
\def \C{\mathbb{C}}

\begin{document}


\baselineskip=17pt


\title[The reflexivity of hyperexpansions]{The reflexivity of hyperexpansions and their Cauchy dual operators}

\author[S. Podder]{Shubhankar Podder}
\address{School of Mathematics\\ Harish-Chandra Research Institute, HBNI\\
Chhatnag Road, Jhunsi\\
Prayagraj (Allahabad) 211019, India}
\email{shubhankar.podder@gmail.com}

\author[D.K. Pradhan]{Deepak Kumar Pradhan}
\address{Statistics and Mathematics Unit\\ Indian Statistical Institute\\
8th Mile, Mysore Road\\
Bangalore 560059, India}
\email{deepak12pradhan@gmail.com}

\date{}

\begin{abstract}
We discuss the reflexivity of hyperexpansions and their Cauchy dual operators. In particular, we show that any cyclic completely hyperexpansive operator is reflexive. We also establish the reflexivity of the Cauchy dual of an arbitrary $2$-hyperexpansive operator. As a consequence, we deduce the reflexivity of the so-called Bergman-type operator, that is, a left-invertible operator $T$ satisfying the inequality
$TT^* + (T^*T)^{-1} \Le 2 I_{\mathcal H}.$
\end{abstract}



\maketitle

\section{Introduction}
Completely hyperexpansive operators were introduced independently by Aleman \cite{A} and Athavale \cite{At}. It has been extensively studied by several authors (see, for example, \cite{SA}, \cite{JS}, \cite{J02}, \cite{ACJS2}, \cite{ACJS1}). It is worth mentioning that the class of completely hyperexpansive weighted shifts is antithetical to that of contractive subnormal weighted shifts, in the sense that the weights of its Cauchy dual are
exactly the weights of contractive subnormal weighted shifts (see \cite[Remark 4]{At}). The present paper investigates the  class of hyperexpansions with a focus on reflexivity. It is to be noted that by a result of Olin and Thompson \cite[Theorem 3]{OT}, any subnormal operator is reflexive. Although the Cauchy dual of a
completely hyperexpansive operator is not necessarily subnormal (refer to \cite[Examples 6.6 and 7.10]{ACJS1}), surprisingly, Proposition \ref{cauchy-dual-ref} ensures the reflexivity of the Cauchy dual of any $2$-hyperexpansive operator.

We set below the notations used throughout this text. Let $\mathbb N$ denote the set of positive integers. Let $\mathbb C$ be the complex plane, while $\mathbb D$ stand for the open unit disk in $\mathbb C$ centered at the origin. All the Hilbert spaces to occur below are complex and separable. For Hilbert spaces $\H$, we use
$B(\H)$ to denote the algebra of bounded linear operators on $\H$. Unless stated otherwise, the Hilbert spaces are infinite-dimensional. The kernel, range, adjoint and spectrum of an operator $T\in B(\H)$ are denoted by  $\mbox{ker}\, T$, $T(\H)$, $T^*$ and $\sigma(T)$, respectively. The symbol $I_{\H}$ is reserved for the identity operator of $B(\H)$. If $F$ is a subset of $\H$, the closure of $F$ is denoted by $\overline{F}$, while the closed
linear span of $F$ is denoted by $\bigvee \{x : x \in F\}$. If $\mathcal N$ is a finite-dimensional subspace of $\mathcal H$, then $\dim \mathcal N$ denotes the vector space dimension of $\mathcal N.$

Let $\mathscr W$ be a subalgebra of $B(\H)$ containing $I_{ \H}$, and let $\mbox{Lat}\,\mathscr W$ be the set of all closed linear subspaces of $\H$ that are invariant under every operator in $\mathscr W$. For $T \in B(\H)$, let $\mbox{Lat}\,T$ denote the set of all  closed linear subspaces of $\H$ that are invariant under $T$.  The set
\begin{center}
AlgLat\,$\mathscr W: =\{T\in {B}(\H): {\rm Lat}\,\mathscr W \subseteq {\rm Lat}\,T\}$
\end{center}
is a subalgebra of $B(\H)$ which contains $\mathscr W$ and WOT-closed. We say that $\mathscr W$ is \emph{reflexive} if $$\mathscr W = \mbox{AlgLat}\,\mathscr W.$$ For $T \in {B}(\H)$, let $\mathscr W_{T}$ (resp. $\mathscr A_{T}$) stand for the WOT-closed (resp. weak*-closed) subalgebra of $B(\H)$ generated by $T$ and $I_{ \H}$. We say $T$ is {\it reflexive} if $\mathscr W_T$ is reflexive. It is easy to see that $\mathscr W_T$ is reflexive if $\mathscr A_T$ is reflexive, and in that case $\mathscr W_T$ equals $\mathscr A_T$. An algebra $\mathscr W$ is said to be {\it super-reflexive} if any unital WOT-closed subalgebra of $\mathscr W$ is reflexive and we call an operator $T$ {\it super-reflexive} if $\mathscr W_T$ is super-reflexive.

Let $\mathcal S$ be a weak*-closed subspace of $B(\H)$.  We say that $\mathcal S$ admits the {\it property }($\mathbb A_1$) if for any weak*-continuous linear functional $\phi$ on $\mathcal S$ there exist vectors $f, g$ in $\H$ such that $\phi (S)=\inp{Sf}{g}$ for all $S \in \mathcal S$. Further, given $r \Ge 1$, if for every weak*-continuous linear functional $\phi$ on $\mathcal S$ and $s>r$ there exist vectors $f, g$ in $\H$ such that $\phi(S)=\inp{Sf}{g}$ for all $S \in \mathcal S$ and $\|f\| \|g\| \Le s \| \phi \|$, then we say that $\mathcal S$ has the {\it property}($\mathbb A_1(r)$). For $r \Ge 1$, we say $T \in B(\H)$ admits the {\it property} ($\mathbb A_1(r)$) (resp. ($\mathbb A_1$)) if $\mathscr A_T$ satisfies the property ($\mathbb A_1(r)$) (resp. ($\mathbb A_1$)). A comprehensive account of these properties and related topics can be found in \cite{BFP}.

\begin{definition}
Let $\H$ be a Hilbert space and $T \in  B(\H)$. Define $B_{n}(T)$ as the following
\begin{equation*}\label{equ}
  B_n(T):=\sum_{p=0}^{n}(-1)^p {{n}\choose {p}}T^{*p}T^p.
 \end{equation*}
 \begin{enumerate}
 \item[(i)]  $T$ is said to be  {\it completely
hyperexpansive} if $B_n(T) \Le 0$ for all $n \in \mathbb N$.
\item[(ii)] For $m \in \mathbb N$,  $T$ is said to be {\it $m$-hyperexpansive} if $B_n(T) \Le 0$ for $n=1,\ldots, m$.
\item[(iii)] For $m \in \mathbb N$, $T$ is said to be {\it $m$-isometric} if $B_m(T)=0$.
 \end{enumerate}
\end{definition}
In case $m=1$, an $1$-hyperexpansive (resp. $1$-isometry) operator $T$ is said to be {\it expansive} (resp. {\it isometry}). The notion of $2$-hyperexpansive operators first appeared in the paper \cite{R} of Richter. He proved in \cite[Lemma 1]{R} that for any $T$ if $B_2(T)\Le 0$, then $B_1(T)\Le 0.$ Hence any $2$-isometric operator $T$ is indeed $2$-hyperexpansive. Also, it is well known that any $2$-isometric operator is completely hyperexpansive (see \cite[Remark 1.3]{SA}). In \cite{S}, Shimorin referred to $2$-hyperexpansive operators as concave operators.

Following \cite{S}, we say that $T \in B(\H)$ is {\it analytic} if $\cap_{n \in \mathbb N}T^n( \H)=\{0\}$. An operator $T$ is called {\it completely non-unitary} if there is no reducing subspace
$\mathcal N$ of $T$ such that $T|_{\mathcal N}$ is a unitary. Note that any analytic operator is completely non-unitary. We say $T$ is {\it finitely multicyclic} if there is a finite, linearly independent subset $W$ of $\H$ such that
\beqn
\H = \bigvee {\{p(T)h : p\in \C[z], h \in W \}}.
\eeqn
In particular, if cardinality of $W$ is 1, then we say that the operator $T$ is {\it cyclic}.


We now briefly recall the
notion of Cauchy dual of a left-invertible
operator, which was first investigated by Shimorin in \cite{S}. An operator $T \in B(\H)$ is {\it left-invertible} if and only if $T^*T$ is invertible, which holds if and only if $T$ is bounded from below. This allows us to define the
{\it Cauchy dual} $T'$ of a left-invertible operator $T$ by $$T':=T(T^*T)^{-1}.$$  For a left-invertible operator $T$, the following can be
seen easily:
\begin{align} \label{equ}
 \left.
  \begin{minipage}{40ex}
\beqn
& & T'^*T=I_{\H},\, T^*T'=I_{\H},\\
& & (T')'=T,\, T'^*T'=(T^*T)^{-1}.
  \eeqn
 \end{minipage}
   \right\}
\end{align}
 A left-invertible operator $T$ is said to be {\it Bergman-type operator} if
 $$TT^* + (T^*T)^{-1} \Le 2 I_{\H}.$$
We now give an outline of this paper. In Section 2, we establish the reflexivity of cyclic completely hyperexpansive operators (Proposition \ref{complete-hyp-ref}). As an application, we deduce that any cyclic $2$-isometry is reflexive. This does not hold in general for cyclic $m$-isometries,  $m>2$. We also discuss some instances in which finitely multicyclic completely hyperexpansive operators are reflexive. In Section 3, we show that the Cauchy dual of $2$-hyperexpansions are reflexive (Proposition \ref{cauchy-dual-ref}). Among various applications, we prove the reflexivity of the Bergman-type operators. Also, we deduce the fact that the Cauchy dual of an $m$-isometry is reflexive for $m=1, 2$; but the result is not true in general for $m > 2.$  However, in case of a contractive or expansive, $m$-isometry its Cauchy dual is always reflexive. Further, we establish the reflexivity of certain power bounded left-invertible operators and study some cases where left-invertible operators and their Cauchy duals are reflexive. We end this paper with some possible directions for further research.

\section{Completely hyperexpansive operators}
In this section, we address the question of the reflexivity for a completely hyperexpansive operator. We focus our attention to the class of cyclic completely hyperexpansive operators primarily due to a result of Aleman, which
asserts that any cyclic analytic completely hyperexpansive operator is unitarily equivalent to the multiplication
operator $M_z$ on a Dirichlet-type space $D_{\mu}$ for some finite positive measure $\mu$ supported on $\overline{\D}$ (see \cite[IV, Theorem 2.5]{A}). We now recall the definition of $D_{\mu}$.
Let $f: \D \rightarrow \C$  be an analytic function of the form $f(z)=\sum_{n=0}^{\infty} \hat{f}(n) z^n$. For a finite positive measure
$\mu$ on $\overline{\D}$, define
 \beqn
 \|f\|^2_{\mu} := \sum_{n=0}^{\infty} |\hat{f}(n)|^2 + \int_{\D} |f'(\zeta)|^2 U_{\mu}(\zeta) dm(\zeta),
 \eeqn
 where $dm$ is the normalized area measure on $\D$, and
 \beqn
 U_{\mu}(\zeta) :=\int_{\D}\mbox{log} \biggl| \frac{1-\overline z \zeta}{\zeta - z}\biggr| \frac{d\mu(z)}{1-|z|^2}+ \int_{\partial \D}
 \frac{1-|\zeta|^2}{|z-\zeta|^2} d\mu(z), \quad \zeta \in \D.
 \eeqn
The Dirichlet-type space $D_{\mu}$ is defined as
\beqn D_{\mu} := \{f: \D \rightarrow \C ~|~ f ~\mbox{is~analytic function such that~}\|f\|_{\mu} < \infty\}. \eeqn

One of the ingredients in obtaining reflexivity of  cyclic completely hyperexpansions is a fact related to the multiplier algebra of scalar-valued reproducing kernel Hilbert spaces. We briefly discuss reproducing kernel Hilbert spaces and some related topics. Most of the fact mentioned below pertaining to reproducing  kernel Hilbert spaces can be found in \cite{PR}.

Let $X$ be a set and $\mathcal{E}$ be a Hilbert space. Let $\H$  be a Hilbert space of $\mathcal E$-valued functions on $X$. We say $\H$ is an {\it $\mathcal{E}$-valued  reproducing  kernel Hilbert space (RKHS) on $X$} provided for every $x \in X$, the evaluation map $E_x :\H \rightarrow \mathcal E$ given by $E_x(f)=f(x)$ is  bounded. The map $k : X\times X \rightarrow B(\mathcal{E})$ defined by $k(x,y):= E_xE_y^*$ is called the {\it $B(\mathcal{E})$-valued reproducing kernel}. In this case, for any $\eta \in \mathcal E$ and $x \in X$, we have $k(\cdot, x)\eta= E_x^*(\eta)$. The space $\mathcal H$ admits the reproducing property:
\beq \label{reproducing-property}
\inp{f(x)}{\eta}_{\mathcal{E}} =\inp{ f} {k(\cdot,x)\eta}_{\mathcal H} \mbox{ for } f \in \H ,\,  x\in X \mbox{ and  } \eta \in \mathcal{E}.
\eeq
We omit the suffix  of the inner product in future usages, whenever the context is clear. In case  $\mathcal{E} = \C$, we say $\H$ is a {\it scalar-valued RKHS} and $k$ to be a {\it scalar-valued reproducing kernel.}

Let $\H$ be a $\mathcal{E}$-valued RKHS on a set $X$. A map $\phi: X \rightarrow B(\mathcal{E})$ is said to be a multiplier of $\H$ if $\phi f \in \H$ for every $f\in \H$, where $\phi f (x) := \phi(x)f(x)$ for $x\in X$. Let $\mathcal M(\H)$ denote the set of all
multipliers of $\H$. As an application of the closed graph theorem, every multiplier $\phi$ induces a bounded linear operator $M_{\phi}:\H\rightarrow \H$ given by $M_{\phi} f=\phi f$, $f \in \H$. It is well known that for any scalar-valued RKHS $\H$, the algebra $\mathscr M(\H):=\{M_\phi: \phi \in \mathcal M(\H) \}$ is
reflexive. Also, in case of an $\mathcal{E}$-valued RKHS, it follows from \cite[Corollary 2.2]{Ba} that $ \mathscr M(\H)$ is reflexive provided the evaluation map $E_x : \H \rightarrow \mathcal{E}$ is either onto or zero for every $x \in X$. It turns out that if $\mathcal E$ is finite-dimensional, then $ \mathscr M(\H)$ is reflexive irrespective of the aforementioned assumptions. The argument is routine and we include a proof for completeness sake.
\begin{lemma}\label{multiplier-reflexivity}
Let $X$ be a set and $\mathcal{E}$ be a finite-dimensional Hilbert space. Let $\H$ be an $\mathcal E$-valued reproducing kernel Hilbert space on $X$ associated with a $B(\mathcal E)$-valued reproducing kernel $k$. Then $\mathscr M(\H)$ is reflexive.
\end{lemma}
\begin{proof}
Given a multiplier $\phi \in \mathcal M (\H),$ it is easy to verify that
$$M_\phi^*k(\cdot, x)\eta=k(\cdot,x)\phi(x)^*\eta, \quad x\in X, ~\eta \in \mathcal{E}.$$
In particular, for each  $x\in X$, $V_x:=\{k(\cdot,x) \eta: \eta \in \mathcal E\} \in \mbox{Lat}\,M_\phi^*$ for all $\phi \in \mathcal M(\H).$ Let $T\in \mbox{AlgLat}\, \mathscr M(\H)$, then $ \mbox{Lat}\,M_\phi^* \subseteq \mbox{Lat}\, T^*$ for every $\phi \in \mathcal M(\H)$ and hence $T^*(V_x) \subset V_x$ for all $x\in X$.
Given $x\in X$, it follows that for every $\eta \in \mathcal E$  there exist $\tilde{\eta}\in \mathcal E$ such that $T^*k(\cdot,x)\eta=k(\cdot,x)\tilde{\eta}$. Now since $\mathcal E$ is
finite-dimensional, there exist a unique $\tilde{\eta} \in E_x(\H)$ such that $T^*k(\cdot,x)\eta = k(\cdot,x)\tilde{\eta}$.
Hence for each $x\in X$, we define
$\tilde{\psi}(x):\mathcal E\rightarrow \mathcal E$ by $\tilde{\psi}(x)(\eta)
:=\tilde{\eta}$.  It is easy to verify that
$\tilde{\psi}(x)$ is linear, and since $\mathcal E$ is finite-dimensional $\tilde{\psi}(x)$ is also bounded. Let
$\psi(x):=\tilde{\psi}(x)^*$, then  for any $f\in \H$, $\eta \in \mathcal{E}$ and $x\in X$,
 \beqn
 \inp{(Tf)(x)}{\eta} &\overset{\eqref{reproducing-property}}=&\inp{ Tf} {k(\cdot,x)\eta}\\
 &=& \inp{f} {T^*k(\cdot,x)\eta}\\
 &=& \inp{f}{k(\cdot,x)\psi(x)^*\eta} \\
 &\overset{\eqref{reproducing-property}}=& \inp{f(x)}{\psi(x)^*\eta}\\
 &=& \inp{\psi(x) f(x)}{\eta}.
 \eeqn
 Hence $T= M_{\psi} \in \mathscr M(\H)$, which completes the proof.
\end{proof}

We recall the definition of complete Nevanlinna-Pick kernel form \cite{S1}. Let $\H$ be a scalar-valued RKHS with a scalar-valued reproducing kernel $k$ on a set $X.$ We say that $k$ has {\it complete Nevanlinna-Pick property} if for $x,y \in X$,
\beqn
k(x,y)-\frac{k(x,w)k(w,y)}{k(w,w)}=B_{w}(x,y)k(x,y),
\eeqn
for some $w\in X$ such that $k(w,w)\neq 0$ and some positive semidefinite function $B_{w}(x,y)$ with $|B_{w}(x,y)| <1$.

In the following proposition we record the known facts \cite[Proposition 2.5(5)]{HN}, \cite[Proposition 2.04, 2.055, 2.09]{BFP}, \cite[Corollary 5.3]{DH}, pertaining to the property $(\mathbb A_1)$ for ready reference.

\begin{proposition}\label{dirct-sum}
The following statements hold.
\begin{enumerate}
\item[(i)] If $\mathscr W$ is a unital WOT-closed subalgebra and $\mbox{AlgLat}\,\mathscr W$ has property $(\mathbb A_1)$, then $\mathscr W$ is super-reflexive.
\item[(ii)] If $\mathcal S$ is any weak*-closed subspace with property $(\mathbb A_1(r))$ for some $r\Ge 1$, and $\mathcal T$ is a weak*-closed subspace of $\mathcal S$, then $\mathcal T$ has property $(\mathbb A_1(r))$.
 \item[(iii)] For some $r\Ge 1$, a direct sum of two unital weak*-closed subalgebras has property $(\mathbb A_1(r))$ if each summand has property $(\mathbb A_1(r))$.
\item[(iv)] If $\mathscr A$ is a unital weak*-closed subalgebra of $B(\H)$ that has property $(\mathbb A_1(r))$ for some $r \Ge 1$, then $\mathscr A$ is WOT-closed; and the weak* and weak operator topologies coincide on $\mathscr A$. Moreover, if $L$ is an invertible operator in $B(\H)$, then $L^{-1}\mathscr A L:=\{L^{-1}AL: A\in \mathscr A\}$ is a unital weak*-closed subalgebra that has property $(\mathbb A_1(r'))$ for some $r' \Ge 1$.
\item[(v)] The multiplier algebra $\mathscr M(\mathcal H)$ of every complete Nevanlinna-Pick kernel has property $(\mathbb A_1(1))$.
\end{enumerate}
\end{proposition}
It is known that any abelian von Neumann algebra has property $(\mathbb A_1(1))$ (see \cite[Proposition 60.1]{CO1}). Hence, by Proposition \ref{dirct-sum}(ii), for any normal operator $T$, $\mathscr A_{T}$ has property
$(\mathbb A_1(1))$. The following is the main result of this section.

\begin{proposition}\label{complete-hyp-ref}
Let $T \in B(\H)$ be completely hyperexpansive with dimension of $\ker T^*$ being 1. Then $T$ is reflexive and admits property $(\mathbb A_1(1))$.
\end{proposition}
\begin{proof}
Since $T$ is  completely hyperexpansive operator, hence 
by \cite[Theorem 3.6]{S}, $T$ admits the following Wold-type decomposition
$$T=T_1 \oplus T_2 \in B(\H_1\oplus \H_2),$$
where $\H_1= \cap_{n \in \mathbb N} T^n(\H)$ and $\H_2 = \bigvee\{T^n x: x\in \mathcal N, n \in \mathbb N \cup \{0\}\}$, with $ \mathcal N:=\mbox{ker}\, T^*$.
Further, $T_1=T|_{\H_1}$ is unitary and $T_2=T|_{\H_2}$ is analytic. Note that, by \cite[Theorem 1]{R}, $T_2$ is cyclic as $\dim \mathcal N=1$ and hence by a theorem of Aleman \cite[IV, Theorem 2.5]{A}, $T_2$ is unitarily equivalent to $M_z$ on $D_{\mu}$ for some finite positive measure $\mu$ on $\overline{\D}$. Furthermore, from \cite[Theorem 1.1]{S1}, it follows that the scalar-valued reproducing kernel $k_{\mu}$ associated with $D_{\mu}$ is
complete Nevanlinna-Pick. Thus, by Proposition \ref{dirct-sum}(v), the algebra $\mathscr M(D_{\mu})$ has property $(\mathbb A_1(1))$ and  $\mathscr M(D_{\mu})$ is reflexive by Lemma \ref{multiplier-reflexivity}. Therefore, by Proposition \ref{dirct-sum}(i), $\mathscr M(D_{\mu})$ is super-reflexive, that is, every unital WOT-closed
subalgebra of $\mathscr M(D_{\mu})$ is reflexive. In particular $\wot{M_z}$, and hence $T_2$, is reflexive. Also, $\mathscr A_{T_2}$ has property $(\mathbb A_1(1))$ by Proposition \ref{dirct-sum}(ii). On the other hand, $T_1$, being a unitary operator, is reflexive and  $\mathscr A_{T_1}$ has property $(\mathbb A_1(1))$ (see the discussion prior to Proposition \ref{complete-hyp-ref}). Since $T_1$ and $T_2$ are reflexive and both $\mathscr A_{T_1}$, $\mathscr A_{T_2}$ have property $(\mathbb A_1(1))$, one may now imitate the proof of \cite[VII, Theorem 8.5, Case 2]{CO} to conclude the reflexivity of $\mathscr A_{T}$, and hence of $T$. Moreover, $\mathscr A_{T}$ has property $(\mathbb A_1(1))$ since each $\mathscr A_{T_i}$ ($i=1, 2)$ has property $(\mathbb A_1(1))$ (see Proposition \ref{dirct-sum}(iii)).
\end{proof}

One part of the proof of reflexivity of a cyclic analytic $2$-isometry $T\in B(\H)$ is implicit in \cite{RS}. Indeed, any such $T$ is unitarily equivalent to $M_z$  acting on $D_{\mu}$ for some finite positive measure $\mu$ supported on $\partial\D$ (see \cite[Theorem 5.1]{R1}). It follows from \cite[Lemma 5.4]{RS} that for any multiplier $\phi \in \mathcal{M}(D_{\mu})$ the multiplication operator $M_{\phi} \in \wot{M_z}$. One can also verify that $\mbox{AlgLat}\,\wot{M_z} \subseteq  \mathscr M(D_{\mu})$ (see, for instance, \cite[Proof of Theorem 4.1]{CPP}) and thus $M_z$ is reflexive. The novelty of our result is that it removes the analyticity assumption and at the same time ensures the property $(\mathbb A_1(1))$ (cf. \cite[Theorem 4]{FH}).

\begin{corollary}\label{2-iso-ref}
If $T$ is a cyclic $2$-isometry in $B(\H)$, then $T$ is reflexive and has property $(\mathbb A_1(1))$.
\end{corollary}

We discuss here one instance in which the reflexivity of (not necessarily cyclic) completely hyperexpansive operators can be ensured. Let $T \in B(\H)$ be a finitely multicyclic completely non-unitary $2$-isometry satisfying the kernel condition, that is, $T^*T(\mbox{ker}\, T^*) \subseteq \mbox{ker}\, T^*$. It follows from \cite[Corollary 3.7]{ACJS2} that $T$ is unitarily equivalent to an orthogonal sum of $n$ unilateral weighted shifts $T_i$, where $n$ is the order of multicyclity of $T$. Since $T_i$'s are cyclic $2$-isometries, each $T_i$ is reflexive and has property $(\mathbb A_1(1))$ (Corollary \ref{2-iso-ref}). Then applying \cite[Theorem 3.6]{S} and using the fact that
any analytic operator is completely non-unitary, we can show that any finitely multicyclic $2$-isometry satisfying the kernel condition is reflexive.

The preceding discussion together with Proposition \ref{complete-hyp-ref} motivates us to the following question:
\begin{question}
Is every completely hyperexpansion in $B(\H)$ reflexive?
\end{question}

\section{Cauchy dual operators}
The main result of this section establishes the reflexivity of the Cauchy dual of certain expansive operators.
Before proceeding to the main result, we recall the following notion. Let $H^{\infty}(\mathbb D)$ be the algebra of all bounded holomorphic functions on $\mathbb D$. A subset  $\Omega$ of $\mathbb C$ is {\it dominating for the algebra $H^{\infty}(\mathbb D)$}  if
$$ \sup_{\lambda \in \D} |f(\lambda)| = \sup_{\lambda \in \Omega\, \cap\, \mathbb D }|f (\lambda)|, \quad
 f \in H^{\infty}(\mathbb D).$$

\begin{proposition}\label{cauchy-dual-ref}
Let $T$ be a $2$-hyperexpansive operator in $B(\H)$ and $T'$ be the Cauchy dual of $T$. Then $T'$ is reflexive and has property $(\mathbb A_1(1))$.
\end{proposition}
\begin{proof}
If $T$ is invertible, then $T$ is a unitary \cite[Remark 3.4]{SA}. In that case, $T'=T$ and the result follows. So, we assume that $T$ is not invertible. Since $T$ is expansive, $T'$ is a contraction. Let us recall that $T'$ admits canonical decomposition
$$T'=T'_a \oplus T'_s \in B(\H_1\oplus \H_2),$$
where $T'_a=T'|_{ \H_1}$ is an absolutely continuous contraction and $T'_s=T'|_{\H_2}$ is a singular unitary operator. It is well known that a contraction is reflexive if and only if its absolutely continuous contractive part is reflexive (see, for example, \cite[Lemma 7.1]{CEP}). We now establish the reflexivity of $T'_a$. As $T'_a$ is an absolutely continuous contraction, $T'_a$ has (Sz.-Nagy-Foias) $H^{\infty}$-functional calculus $\Phi_{T'_a}:H^{\infty}(\mathbb D) \rightarrow \mathscr A_{T'_a}$ (see \cite[Theorem 4.1]{BFP}). Since $T$ is not invertible, $\sigma(T')=\overline{\mathbb D}$ by \cite[Lemma 2.14(ii)]{C} and thus $\sigma(T'_a)=\overline{\mathbb D }$. Therefore, $\sigma(T'_a)\cap \mathbb D$ is dominating for $H^{\infty}(\mathbb D)$. It follows then from \cite[Proposition 4.6]{BFP} that the $H^{\infty}$-functional calculus $\Phi_{T'_a}$ is an isometric isomorphism and a weak* homeomorphism. Hence $T'_a$ is reflexive by \cite{BC}.

By \cite[Theorem 1.2]{B}, $\mathscr A_{T'_a}$ satisfies property $(\mathbb A_1(1))$. Since $T'_s$ is a unitary operator, $\mathscr A_{T'_s}$ has property $(\mathbb A_1(1))$ (see the discussion prior to Proposition \ref{complete-hyp-ref}). Therefore, by Proposition \ref{dirct-sum}(iii), $\mathscr A_{T'}$ has property $(\mathbb A_1(1))$. This completes the proof.
\end{proof}

\begin{remark}
The reflexivity of $T'_a$ can also be obtained from \cite[Theorem 5]{FH}. There the proof is based on the fact that every von Neumann operator is reflexive.
\end{remark}

Bergman-type operators were studied by Shimorin \cite{S} in the context of wandering subspace problem. As an immediate application of the above proposition, we now establish the reflexivity of such operators.

\begin{corollary}\label{coro2}
Let $T \in B(\H)$ be a Bergman-type operator. Then $T$ is reflexive and has property $(\mathbb A_1(1))$.
\end{corollary}
\begin{proof}
Suppose $T \in B(\H)$ is a Bergman-type operator. It was noted in the proof of \cite[Theorem 3.6]{S} that  $T'$  is $2$-hyperexpansive. Now by applying Proposition \ref{cauchy-dual-ref}, we deduce the reflexivity of $(T')'=T$ and property $(\mathbb A_1(1))$.
\end{proof}

\begin{remark}
Let $T$ be a $2$-hyperexpansive operator in $B(\H)$ and let $T'$ be the Cauchy dual of $T$. It was observed in the proof of \cite[Theorem 2.9]{C} that $TT'$ is similar to an isometry. Also, one can easily show that $T'T$ is similar to an isometry. Since any isometry is reflexive (see \cite{D}) and reflexivity is invariant under similarity (see \cite[Section 57, pp. 323]{CO1}), it follows that $TT'$ and $T'T$ are reflexive. Moreover, since an isometry has property $(\mathbb A_1(1))$, so by applying Proposition \ref{dirct-sum}(iv) we deduce that $TT'$ (resp. $T'T$) has property $(\mathbb A_1(r))$ (resp. $(\mathbb A_1(r'))$) for some $r\Ge 1$ (resp. $r'\Ge 1$).
\end{remark}

As discussed earlier, any $2$-isometry is $2$-hyperexpansive. Therefore by Proposition \ref{cauchy-dual-ref}, Cauchy-dual of any $2$-isometry is reflexive. It turns out that any $m$-isometry is left-invertible (see \cite[I. Lemma 1.21]{AS}). This gives rise to  a natural question, whether the Cauchy dual of any $m$-isometry is reflexive
or not. This has a negative answer for $m=3$. In this regard, we discuss below an example of a cyclic $3$-isometry $T$ from \cite[III. pp 406]{AS}, for which neither $T$ nor $T'$ is reflexive.

\begin{example}\label{example}
 Let us consider the following operator
$$ T :=
  \bigoplus_{n=1}^{\infty}\begin{bmatrix}
    \alpha_n & c \\
    0 & \alpha_n
  \end{bmatrix} ~ \mbox{ on }  \mathcal H := \bigoplus_{n=1}^{\infty} \H_{n},$$
where each $\H_{n}= \mathbb C^2$, $c>0$ and $(\alpha_n)_{n \in \mathbb N}$ is a sequence of unimodular complex numbers which do not contain any of its accumulation points. It is easy to verify that
\beqn
 T^*T = \bigoplus_{n=1}^{\infty}\begin{bmatrix}
    1 & \overline{\alpha_n}c  \\
    \alpha_n c & 1+c^2
  \end{bmatrix} ~ \mbox{ and }
 T' = \bigoplus_{n=1}^{\infty}\begin{bmatrix}
    \alpha_n & 0  \\
    -\alpha_n^2c & \alpha_n
  \end{bmatrix}.
\eeqn
Let $\{e_1,e_2\}$ be the standard basis of $\mathbb C^2$. For $j=1,2$, let $$\ x_j^n = (0,\ldots,0,e_j,0,\ldots),$$
where $e_j$ is in the $n^{{th}}$ position. Note that $\{x_j^n: j =1,2, ~\mbox{and } n \in \mathbb N\}$ is an orthonormal basis for  $\mathcal H$. A routine verification shows that
$$ \mbox{Lat}\,T = \big\{\{0\},\{ \vee \{x_{1}^n\} : n\in \mathbb N\}, \{ \vee \{x_{1}^m,x_{1}^n\} : m,n\in \mathbb N\},\cdots,\mathcal H\big\},$$
$$ \mbox{AlgLat }\mathscr W_T = \Big\{ \bigoplus_{n=1}^{\infty} \begin{bmatrix}
    p_n & q_n  \\
    0 & r_n
  \end{bmatrix} : p_n,q_n,r_n \in \mathbb C \Big\}.
$$
It is easy to show that any $A \in \wot{T}$ is of the form
\beqn
 \bigoplus_{n=1}^{\infty}\begin{bmatrix}
    t_n & q \\
    0 & t_n
  \end{bmatrix} \mbox{ for some } t_n,q \in \C.
\eeqn
Thus $\wot{T}$ is a proper subset of $\mbox{AlgLat}\,{\wot{T}}$ and hence $T$ is not reflexive. A similar computation shows that $T'$ is not reflexive.
\end{example}

Observe that the $3$-isometry considered in Example \ref{example} is neither expansive nor contractive. In the following proposition we show that the Cauchy dual $T'$ of an $m$-isometry $T$ is reflexive provided $T$ is either expansive or contractive.

\begin{proposition}\label{sameer}
Let $T\in B(\mathcal H)$ be an $m$-isometry and let $T'$ be the Cauchy dual of $T$. If $T$ is either contractive or expansive, then $T'$ is reflexive and has property $(\mathbb A_1(1))$.
\end{proposition}
\begin{proof}
If $m=1$, then $T$ is an isometry. In that case, $T'=T$ and the result follows as any isometry is reflexive (see \cite{D}) and has property $(\mathbb A_1(1))$. Let $m > 1$.

Case I. Suppose $T$ is a contraction. It follows from \cite[Lemma 2.4]{C1} that $T$ is expansive and hence $T$ is an isometry. Thus, $T$ as well as $T'$ is reflexive and both of them have property $(\mathbb A_1(1))$.

Case II. Suppose $T$ is an expansion. It is well known for any $m$-isometry $T$,
$\sigma(T) \subseteq \overline{\D}$ (see \cite[I. Lemma 1.21]{AS}). Note that for any $\lambda \in \D \setminus \{0\} $,
$$T'^* - \lambda I_{\H}\overset{\eqref{equ}}= T'^*(I_{\H} - \lambda T)=-\lambda T'^*(T-(1/\lambda) I_{\H}). $$
If $T$ is not invertible, then $T'$ is not invertible. In that case, $\overline{\D} \subseteq \sigma(T')$. In addition, the contractivity of $T'$ implies $\sigma(T')=\overline{\D}$. Now it follows from the argument
of the proof of Proposition \ref{cauchy-dual-ref} that for any non-invertible expansive $m$-isometry $T$, the Cauchy dual $T'$ is reflexive and has property $(\mathbb A_1(1))$. On the other hand, if $T$ is invertible, then $T^{-1}$ is a contractive $m$-isometry. Therefore, by Case I, $T^{-1}$ is an isometry  and hence $T$ is a unitary. This completes the proof.
\end{proof}

Although $T'$ in Proposition \ref{sameer} is reflexive and every contractive or invertible expansive, $m$-isometry $T$ is reflexive; we do not have any conclusive evidence about the reflexivity of
$T$ in general. This situation does not occur in case of finite-dimensional Hilbert spaces.

\begin{proposition}\label{finite-ref}
Let $T$ be an invertible operator acting on a finite dimensional Hilbert space $\H$. Then $T$ is reflexive if and only if $T'$ is reflexive.
\end{proposition}
\begin{proof}
Note that $T'=(T^{-1})^*.$  If $\lambda$ is an eigenvalue of $T$ corresponding to a generalized eigenvector $v$, then $1/{\lambda}$ is the  eigenvalue of $T^{-1}$ corresponding to the generalized eigenvector $v$. It is easy to verify that numbers and sizes of Jordan blocks corresponding to $\lambda$ and $1/{\lambda}$ are same for $T$ and $T^{-1}$ in their respective Jordan canonical forms. Now applying the result of Deddens and Fillmore \cite{DF} (see also \cite[Theorem 57.2]{CO1}), $T$ is reflexive if and only if $T^{-1}$ is reflexive. Since reflexivity is preserved under adjoint operation, therefore $T$ is reflexive if and only if $T'$ is reflexive.
\end{proof}

In the following proposition we give a class of left-invertible operators for which both $T$ and $T'$ are reflexive (cf. \cite[Corollary 4.5]{CPP}, \cite[Corollary 30]{DPP}). Recall that an operator $T\in B(\H)$ is said to be {\it power bounded} if sup$_{n\in \mathbb N}\|T^n\| <\infty$.

\begin{proposition}\label{power-bdd}
Let $T \in B(\H)$ be a left-invertible operator. If both $T$  and $T'$ are power bounded, then $T$ and $T'$ are
reflexive. Moreover, $T$ (resp. $T'$) has property $(\mathbb A_1(r))$ (resp. $(\mathbb A_1(r'))$) for some $r \Ge 1$ (resp. $r' \Ge 1$).
\end{proposition}
\begin{proof}
Suppose that $T$ and $T'$ are power bounded operators. Then there exist positive numbers $K_1$ and $K_2$ such that
$\|T^n\| \Le K_1$ and  $\|T'^n\| \Le K_2$ for all $n\in \mathbb N$. Now for any $x\in \H$,  $\|T^n x\| \Le K_1\|x\|$ and $$\|x\|\overset{\eqref{equ}}=\|T'^{*n}T^nx\|\Le\|T'^{*n}\|\|T^nx\|\Le K_2\|T^nx\|.$$
Thus $(1/K_2) \|x\| \Le \|T^n x\| \Le K_1\|x\|$ for $x \in \H$. We now deduce from Proposition 1.15 and the discussion following Corollary 1.16 in \cite{K}, $T$ is similar to an isometry. Then from the reflexivity of isometry (see \cite{D}) and invariance of reflexivity under similarity, we conclude that $T$ is reflexive. Further, since an isometry has property $(\mathbb A_1(1))$, so by applying Proposition \ref{dirct-sum}(iv) we deduce that $T$ has property $(\mathbb A_1(r))$ for some $r\Ge 1$. Similarly one can show that $T'$ is reflexive and has property $(\mathbb A_1(r'))$ for some $r'\Ge 1$.
\end{proof}

For an expansive power bounded operator $T$, it is evident from \eqref{equ} $T'$ is also power bounded. Hence the following corollary is immediate from the above proposition.

\begin{corollary}\label{coro3}
Let $T \in B(\H)$ be expansive. If $T$ is power bounded, then both $T$ and $T'$ are reflexive. Moreover, $T$ (resp. $T'$) has property $(\mathbb A_1(r))$ (resp. $(\mathbb A_1(r'))$) for some $r \Ge 1$ (resp. $r' \Ge 1$).
\end{corollary}

\section{Concluding Remarks}
We conclude this note with some possible directions for further investigations. We would like to draw the reader's attention to the recent work \cite{EMZ}, where the notion of the Cauchy dual operator has been generalized to the closed range operators. It would be interesting to know counter-parts of the results in this paper for closed range
operators. Although we have obtained reflexivity of any cyclic completely hyperexpansive operator (Proposition \ref{complete-hyp-ref}), we do not know whether an arbitrary completely hyperexpansive operator is reflexive. In view of the reflexivity of the Cauchy dual of a $2$-hyperexpansion (Proposition \ref{cauchy-dual-ref}), one may explore the possibility to relate $\mbox{Lat}\,T$ with $\mbox{Lat}\,T'$, $\wot{T}$ with $\wot{T'}$ and $\mbox{AlgLat}\,{\wot{T}}$ with $\mbox{AlgLat}\,{\wot{T'}}$ for a left-invertible operator $T$. Since $(T')'=T$, these relations may play a decisive role in deriving reflexivity of an arbitrary completely hyperexpansive operator. Indeed, the inclusion of $\mathscr W_{T}$ in $\mathscr W_{T'}$ (or vice-versa) implies the reflexivity of both $T$ and $T'$.

\begin{proposition}\label{prop2}
Let $T \in B(\H)$ be left-invertible and let $T'$ be the Cauchy dual of $T$. If $\mathscr W_{T} \subseteq \mathscr W_{T'}$, then both $T$ and $T'$ are reflexive.
\end{proposition}
\begin{proof}
Assume that $\mathscr W_{T} \subseteq \mathscr W_{T'}$, then $TT'=T'T$. Now multiplying both sides from left by $T'^*$ gives $T'\overset{\eqref{equ}}=(T^*T)^{-1}T$. Hence $(T^*T)T=T(T^*T)$, that is, $T$ is quasinormal and by \cite[Theorem 2]{W} $T$ is reflexive. It is easy to verify that $T$ is quasinormal if and only if $T'$ is
quasinormal. This completes the proof.
\end{proof}

Note that the converse of the above proposition is not true in general. Indeed,  such examples  are ample in the class of weighted unilateral shifts  operators (denoted by $S_{w}$) on  $l^2(\mathbb N).$ As noted in the proof if $\mathscr W_{S_{w}} \subseteq \mathscr W_{S_{w}'}$ then $S_{w}$ is necessarily quasinormal. But it is easy to verify from the definition that a weighted shift is quasinormal if and only if it is a constant multiple of the (unweighted) unilateral shift. Moreover in this case $\mathscr W_{S_{w}} = \mathscr W_{S'_{w}}.$ But the class of reflexive weighted shifts is strictly bigger than unilateral shifts (see \cite{FH} for a large class of examples). 

Recently, reflexivity of certain classes of weighted shifts on directed trees was studied in \cite{BDPP, DPP}. A result in \cite{DPP} suggests  another class of left-invertible operators $T$, for which both $T$ and $T'$ are reflexive. Indeed, it follows from \cite[Corollary 30]{DPP}, for a left-invertible weighted shift $T$ on rooted
directed tree, both $T$ and $T'$ are reflexive. In view of Propositions \ref{finite-ref}, \ref{power-bdd}
and the discussion above, one may ask whether the reflexivity of a left-invertible operator $T$ implies the reflexivity of the Cauchy dual $T'$. Nevertheless, it is interesting to study the class of left-invertible operators for which reflexivity is preserved under Cauchy dual operation.

\subsection*{Acknowledgements}
A part of this paper was written while the first author visited the Department of Mathematics and Statistics, IIT Kanpur. He expresses his gratitude to the faculty and the administration of this unit for their warm hospitality. The authors would like to thank Sameer Chavan for his continual support and  encouragement throughout the preparation of this paper. In particular, Case II in Proposition \ref{sameer} was pointed out by him.


\end{document}